\documentclass[12pt,reqno]{amsart}

\usepackage{amsmath,amsfonts,amsthm,amssymb,array,cite,prettyref}
\usepackage[all]{xy}

\newcommand{\sig}{\sigma}

\newcommand{\vphi}{\varphi}

\newcommand{\pref}[1]{\prettyref{#1}}
\newcommand{\Cal}[1]{\mathcal{#1}}
\newcommand{\bb}[1]{\mathbb{#1}}

\newcommand{\mfrak}[1]{\mathfrak{#1}}

\newcommand{\rcomp}{\backslash}

\newcommand{\h}{\textup{H}}

\newcommand{\ang}[1]{\langle #1 \rangle}
\newcommand{\paren}[1]{\left( #1 \right)}

\newcommand{\abs}[1]{\left| #1 \right|}

\newcommand{\Aut}{\textup{Aut}}
\newcommand{\Gal}{\textup{Gal}}

\newcommand{\bbp}{\mathbb{P}}
\newcommand{\Lie}{\mathrm{Lie}}
\newtheorem{thm}{Theorem}
\newtheorem{lma}[thm]{Lemma}
\newtheorem{cor}[thm]{Corollary}
\newtheorem{prop}[thm]{Proposition}

\newtheorem{rmk}[thm]{Remark}

\theoremstyle{definition}

\newrefformat{dia}{\textup{(\ref{#1})}}
\newrefformat{eq}{\textup{(\ref{#1})}}
\newrefformat{map}{\textup{(\ref{#1})}}

\begin{document}
\title[Rank gain of Jacobian varieties]{Rank gain of Jacobian varieties over finite Galois extensions}
\author{Bo-Hae Im \and Erik Wallace}
\address{
  Department of Mathematical Sciences, KAIST, 
  291 Daehak-ro, Yuseong-gu, 
  Daejeon, 34141, 
  South Korea}
\email{bhim@kaist.ac.kr}
\address{
Department of Mathematics,
University of Connecticut,
341 Mansfield Road U1009,
Storrs, Connecticut 06269-1009,
United States}

\email{erik.wallace@uconn.edu}
\date{\today}
\thanks{ Bo-Hae Im was supported by Basic Science Research Program through the National Research Foundation of Korea (NRF) funded
        by the Ministry of Education, Science and Technology (NRF-2014R1A1A2053748).}

\begin{abstract}
Let $K$ be a number field, and let $\Cal{X}\to\bb{P}^1_K$
be a degree $p$-covering branched only at 0, 1, and $\infty$. If $K$ is a field containing
a primitive $p$-th root of unity then the covering of $\bb{P}^1$ is Galois over $K$, and if
$p$ is congruent to $1 \mod 6$, then there is an automorphism $\sig$ of $\Cal{X}$ which cyclically
permutes the branch points. Under these assumptions, we show that the Jacobian varieties of both $\Cal{X}$ and $\Cal{X}/\ang{\sig}$
gain rank over infinitely many linearly disjoint cyclic degree $p$-extensions of $K$.  
We also show the existence of an infinite family of elliptic
curves whose $j$-invariants are parametrized by a modular function on $\Gamma_0(3)$ and that gain rank over infinitely many cyclic degree $3$-extensions of $\bb{Q}$.
\end{abstract}
\maketitle

\section{Introduction}
The construction in this paper is inspired largely by
a paper of Elkies on the Klein quartic \cite{Elkies}, however
we have been able to prove much more general results.  The general
setup is the following. Let $\Cal{X}$ and $\Cal{Y}$ be curves defined
over a number field $K$, and suppose we have the following diagram
\begin{equation}\label{dia:1}
\xymatrix{
&\Cal{X}\ar[dl]\ar[dr]&\\
\Cal{Y}&&\bb{P}^1_K
}
\end{equation}
where both maps are defined over $K$ and surjective, and the map
$\Cal{X}\to\bb{P}^1_K$ has degree $d>1$.
The strategy is to lift $K$-rational points $P$ of $\bb{P}^1_K$
to points $Q$ on $\Cal{X}$, which in general will lie in an extension $L/K$.
By Hilbert's irreducibility theorem it can be shown that $L/K$ will usually be a degree $d$
extension.

We are especially interested in the case where $L/K$ is a Galois extension.
If the map $\Cal{X}\to\bb{P}^1_K$ is itself Galois (over $K$) with group $G$,
then the degree $d$ extensions $L/K$ that are obtained from Hilbert's irreducibility theorem
are also Galois with group isomorphic to $G$.  But there are sometimes other ways
of getting the extensions $L/K$ to be Galois without $\Cal{X}\to\bb{P}^1_K$
being Galois itself as will be seen in the proof of Theorem~\ref{thm2}.

As for the maps $\Cal{X}\to \Cal{Y}$, such maps can be obtained by taking quotients by
a subgroup of the automorphism group of $\Cal{X}$ or, in the case where $\Cal{Y}$
is an elliptic curve and $\Cal{X}$ is a suitable modular curve, they can be obtained
from the modularity theorem. In a sense the case of the Klein quartic constitutes an
example of both.

Once infinitely many points on $\Cal{X}$ and $\Cal{Y}$ have been produced,
all lying in different degree $d$ extensions of $K$, they can then be used to construct
corresponding points on the Jacobian varieties of $\Cal{X}$ and $\Cal{Y}$. Then by the generalization
of a lemma of Silverman \cite{Silverman}, it can be shown that the Jacobian varieties of $\Cal{X}$ and
$\Cal{Y}$ each gain rank over infinitely many extensions $L/K$. If we are not concerned
with the extensions $L/K$ being Galois, it is a relatively easy matter to show that the rank
gains over infinitely many extensions $L/K$ for infinitely many different degrees as indicated by
Proposition~\ref{thm4}, but with the Galois condition added the problem becomes more difficult
and essentially is related to the inverse Galois problem. In particular for elliptic curves
it is not well understood whether an elliptic curve can gain rank over cyclic Galois extensions
of a number field $K$, outside of some special cases. As the following theorem shows, the Klein
quartic gives us rank gain for certain Elliptic curves over cyclic degree $7$ extensions of $\bb{Q}(\zeta_7$)
where $\zeta_7$ is a primitive 7-th root of unity.

\begin{thm}\label{thm1}
Let $p$ be a prime. Let $\Cal{X}$ be a Riemann surface of genus $g>1$ and let
\[
 \vphi: \Cal{X}\to \bb{P}^1_K
\]
be a degree $p$ Galois covering ramified only at 0, 1, and $\infty$.
If $p\equiv 1\mod 6$, then there is an automorphism $\sig$ of $\Cal{X}$
that cyclically permutes the points 0, 1, and $\infty$, and
there are infinitely many linearly disjoint degree $p$-extensions $L/K$
over which the Jacobian variety of $\Cal{Y}=\Cal{X}/\ang{\sig}$ gains rank.
\end{thm}

Note that for the covering $\vphi$ to be Galois over $K$, it is required to have $K$ to contain $\bb{Q}(\zeta_p)$
where $\zeta_p$ is a primitive $p$-th root of unity. Theorem~\ref{thm1} also has the following Corollary,
which is an immediate consequence of the linear disjointness property.
\begin{cor} Following the same notations as in Theorem~\ref{thm1},
let $J_\Cal{X}$ and $J_\Cal{Y}$ be the Jacobian variety of $\Cal{X}$ and $\Cal{Y}$ respectively. Then
$J_\Cal{X}$ and $J_\Cal{Y}$ have infinite rank over  $K[p]^{ab}$, i.e. the maximal abelian extension $K[p]^{ab}$ of $K$
in the compositum $K[p]$ of all degree $p$-extensions.
\end{cor}
For $p=7$, the curve $\Cal{X}$ is the Klein quartic, which is also known to be the modular curve $X(7)$. It follows,
that we not only have rank gain for the quotient curve, which Elkies \cite{Elkies} denotes by $E_k$, but also
for any curve in the same isogeny class as $E_k$ and more generally for any elliptic curve with conductor 49. However,
in all cases the degree $7$ extensions $L/K$ obtained are only Galois extensions if $K$ contains $\bb{Q}(\zeta_7)$

From a slightly different perspective, it can
be sometimes shown that the extensions $L/K$ in Proposition~\ref{thm3}
are Galois even when  the covering of $\bb{P}_K^1$ is not.
In particular, we prove the following theorem for the genus $1$-case.

\begin{thm}\label{thm2}
 There is an infinite family of elliptic curves over $\bb{Q}$ such
 that each member gains rank over infinitely many linearly disjoint
 cyclic degree $3$-extensions of $\bb{Q}$.
\end{thm}

Here, the infiniteness of the family means that the $j$-invariants of elliptic curves in the family are distinct.
Since an infinite family is obtained, one way of looking at the proof is in terms of a surface, the fibers of which
are curves $\Cal{X}$ with maps to $\bb{P}^1$.  When viewed in this way, the curves $\Cal{X}$ that occur in the
construction can be seen to have genus 4. Recently, Professor Masanobu Kaneko pointed out to the first author that the $j$-invariant
constructed in Theorem~\ref{thm2} is paramerized by a modular function on $\Gamma_0(3)$ whose more details are given at the end of the paper, hence it also can
be viewed in terms as modularity just like the case of the Klein Quartic in Theorem \ref{thm1}.

It is worthwhile to compare Theorem~\ref{thm2} to Theorem~C in \cite{Kuwata1} which shows that if an elliptic curve over $\bb{Q}$
has at least 6 rational points, then it gains rank over infinitely many cyclic degree $3$-extensions of
$\bb{Q}$.  The proof of theorem~\ref{thm2} in this paper provides a family containing curves with rank zero over $\bb{Q}$ and
rational torsion isomorphic to $\bb{Z}/3\bb{Z}$, hence Theorem~C \cite{Kuwata1} does not apply to those curves.  On the other
hand all curves in the family found in Theorem~\ref{thm2} have non-trivial rational 3-torsion, hence Theorem~C in \cite{Kuwata1}
applies to any Elliptic curve with positive rank over $\bb{Q}$ and trivial torsion, but none of those curves are included in
our proof of Theorem~\ref{thm2}. As a consequence not only are the methods used to obtain the theorems completely, but there
are elliptic curves that are known by one of them to grain rank, but not by the other.

The next proposition provides the basic machinery used to prove \ref{thm1} and \ref{thm2} above.

\begin{prop}\label{thm3}
Let $\Cal{X}$ be a smooth irreducible curve of genus $g>1$,
defined over a number field $K$, and let
\[
 \Cal{X}\to \bb{P}^1_K
\]
be a covering map of degree $d>1$. Then  there exist infinitely
many linearly disjoint extensions $L/K$ of degree $d$ over which the
Jacobian variety of $\Cal{X}$ gains rank. Additionally, if the covering is
Galois with group $G$, than there are infinitely many such extensions
which are Galois with group $G$.
\end{prop}

It has already been pointed out that maps can be obtained from $\Cal{X}\to\Cal{Y}$
by taking a quotient, or in the case where $\Cal{X}$ is a modular curve maps can be
obtained to elliptic curves $E$ by the modularity theorem.  But there is another option:
given a morphism from the Jacobian variety of $\Cal{X}$ to another abelian variety $A$,
all results can be extended to the abelian variety. Also if we are not concerned
with the extensions $L/K$ being Galois, then it is relatively easy to obtain the
following generalization of Theorem~1 in \cite{Costa}:
\begin{prop}\label{thm4}
 Let $\Cal{X}$ be a smooth irreducible curve of genus $g>1$,
defined over a number field $K$. Then there is a constant $N$
depending on $\Cal{X}$ but not on $K$ such that for every positive integer
$d\geq N$, there exist infinitely many linearly disjoint
degree $d$-extensions $L/K$ over which the Jacobian variety of $\Cal{X}$ gains rank.
\end{prop}

It is important to note that this generalization improves on two details:
\begin{enumerate}
 \item For Theorem~1 in \cite{Costa}, it is assumed that the curve is bi-rationally
 equivalent over $K$ to a plane curve of the form $f(x)=g(y)$, where the degrees of $f$ and $g$
 are co-prime.  This assumption excludes hyper-elliptic curves $y^2=f(x)$ for which the degree of $f$ is even,
 and it excludes curves $f(x,y)=0$ for which the variables cannot be separated. By contrast the proof
 of Proposition \ref{thm4} is done generally enough so as to cover these cases as well.
 \item For Theorem~1 in \cite{Costa} the degree is assumed to be prime $p$, whereas in Proposition \ref{thm4}
 the degree may be composite.
\end{enumerate}
The constant $N$ is explicit as it is shown in its proof. Even with these improvements however,
the biggest shortcoming of Proposition~\ref{thm4} is the fact that the extensions obtained
are generally not Galois, and an analogous statement of the result for Galois coverings remains an open problem.

\

{\bf Acknowledgments.} This project began as a collaboration with Neeraj Kashyap who made some key
contributions to the construction in the early stages, but later became
otherwise occupied.  The authors would also like to thank Taylor Dupuy
for suggesting the work of Ledet, and Michael Larsen for suggesting
generalizations to our results, which have been incorporated in
the current version. The authors would also like to thank Professor Masanobu Kaneko for making the observation about eta functions as in Remark~\ref{eta}.

\

\section{Proofs}

To prove Proposition~\ref{thm3} and  Theorem~\ref{thm1} we will need the
following lemma, originally proven by Silverman in the genus 1 case \cite{Silverman}.
We include a self-contained proof cause due to the lack of its presence in the literature.

\begin{lma}\label{lma3}
 Let $K$ be a number field, and $A$ be an abelian variety
 defined over $K$. Then
 for every positive integer $d$, we have
\[
  \abs{\bigcup_{[L:\bb{Q}]\leq d} A(L)_{tor}}<\infty,
\]
where $A(L)_{tor}$ denotes the set of torsion points of $A$ that are $L$-rational.
\end{lma}
The general proof is similar, but it appears nowhere in print and
so we briefly give the proof here.
\begin{proof}
Let $\mfrak{p}_1$ and $\mfrak{p}_2$ be two prime ideals of $K$ with
distinct residue characteristics $p_1$ and $p_2$ respectively, at which $A$ has good reduction.
Let $L/K$ be an extension such that $[L:K]\leq d$, and let $\mfrak{P}_1$ and
$\mfrak{P}_2$ be two primes above $\mfrak{p}_1$ and $\mfrak{p}_2$ respectively. Then there exists an injective map
\[
 A_m(L)\to A_m(\Cal{O}_L/\mfrak{P}_i)
\]
where $A_m$ denotes the $m$-torsion of $A$ (see for example \cite[Ch.~7,~Proposition~3]{Bosch}).
By the Lang-Weil bound we have
\[
 |A_m(\Cal{O}_L/\mfrak{P}_i)|\ll N_{L/\bb{Q}} \mfrak{P}_i
 \ll \paren{N_{K/\bb{Q}}\mfrak{p}_i}^d
\]
where the implied constants depend only on $A$. For an arbitrary positive integer $m$,
if $m_i$ denotes the largest factor of $m$ not divisible by $p_i$, then by
the structure theorem of finite abelian groups we have
\[
 |A_m(L)|\leq |A_{m_1}(L)|\cdot |A_{m_2}(L)|\ll \paren{N_{K/\bb{Q}}\mfrak{p}_1\mfrak{p}_2}^d.
\]
Since the right hand side of this estimate does not depend on $m$ or $L$, and since number
of options for $\Cal{O}_L/\mfrak{P}_i$ is finite (up to isomorphism),
this completes the proof.
\end{proof}

We will also need the following lemma which is essentially a direct application
of Hilbert's irreducibility theorem.

\begin{lma}\label{lma1}
Let $\Cal{X}$ be a smooth irreducible curve of genus $g>0$ defined
over a number field $K$, and let
\[
 \Cal{X}\to \bb{P}^1_K
\]
be a covering map of degree $d>1$. Then there exist infinitely
many linearly disjoint extensions $L/K$ of degree $d$ such that
$\Cal{X}$ has an $L$-rational point. Additionally, if the covering is
Galois with group $G$, then there are infinitely many such extensions $L/K$
which are Galois with group $G$.
\end{lma}

\begin{proof}
In the Galois case this is  \cite[Corollary~3.3.4]{Serre}.  For the general
case, let $\Cal{Y}$ be the Galois closure of $\Cal{X}$ such that the covering
\[
 \pi:\Cal{Y}\to \bb{P}^1
\]
has group $G$, and let $H_0\leq G$ be a subgroup such
that $\Cal{Y}/H_0\cong \Cal{X}$. If $L/K$ is an arbitrary finite extension, then
\[
 A_L=\bigcup_{H<G} \pi_H(\Cal{X}/H)(L)
\]
is thin with respect to $L$, where the union is over proper subgroups $H$ of $G$,
and $\pi_H$ is the natural morphism to $\bb{P}^1$ induced by taking the quotient by $H$.
This set has the property that if $P\notin A_L$ then $P$ lifts to a point
 $Q\in\Cal{X}(M)$ where $M/L$ is a degree $d$-extension (the proof is analogous
 to that of \cite[Prop. 3.3.1]{Serre}).

 By  \cite[Prop. 3.2.1]{Serre} the set $A:=A_L\cap K$ is thin with respect to $K$,
 so by applying the above argument to $P\in \bb{P}^1(K)\rcomp A$ we obtain
 the result by induction by taking $L$ to be the compositum of all previously
 obtained degree $d$-extensions.
\end{proof}

\begin{rmk}\label{Special}
In the special case of a Galois covering ramified only at 0, 1 and $\infty$
and of degree $p$, where $p$ is an odd prime, we obtain an equation
\[
 y^p=x^r(x-1)^s
\]
as an affine model, where $r$, $s$ and $r+s$ are relatively prime to $p$.
When $K=\bb{Q}$ it can  be actually shown that 0, 1, and $\infty$
are the only exceptions.  This is done as follows.  Take $x=\frac{a}{b}$, where $a,b$
are rational integers.  By the pairwise relative primality of $a,b,a-b$, and
unique factorization, it can be shown that $a,b,a-b$ each must be a $p$th
power, say $A^p,B^p,C^p$ in $\bb{Z}$ respectively.  Then we must have an integer
solution to
\[
 A^p=B^p+C^p,
\]
by Fermat's Last Theorem we know that the only solutions are trivial,
and the trivial solutions correspond to the points $0$, $1$, and $\infty$.
\end{rmk}

\begin{prop}\label{disjJac}
Let $\Cal{X}$ be a smooth irreducible curve.
 Suppose there exist a finite subgroup $G\subseteq\Aut(\Cal{X})$ and a nontrivial
 abelian variety $\Cal{A}$ satisfying the following conditions:
\begin{enumerate}
\item there exists a morphism $f:\Cal{X}\rightarrow \Cal{A},$ and
\item there exists a homomorphism $\phi: G\rightarrow \Aut(\Cal{A})$ such that\\
$f(\sig(x))-f(x)=\phi(\sig)$ for all $\sig\in G$ and for all  $x\in \Cal{X}$.
\end{enumerate}
Then if $f(\Cal{X})$ generates $\Cal{A}$, then $\dim(\Cal{A})\leq g(\Cal{X}/G)$,
where $g(\Cal{X}/G)$ is the genus of $\Cal{X}/G$.
\end{prop}
\begin{proof} We may take a pair $(\Cal{X}, G)$ with $|G|$ minimal among all such pairs
satisfying conditions (1) and (2). Let $H$ be a subgroup of $G$ generated by all elements $h\in G$
satisfying that there exists $x\in \Cal{X}$ such that $h(x)=x$.
Then $H$ is a normal subgroup of $G$. Indeed, for $h\in H$ and $\sig\in G$,
there exists $x\in \Cal{X}$ such that $h(x)=x$ and so for $\sig^{-1}(x)\in \Cal{X}$,
$(\sig ^{-1}h\sig) (\sig^{-1}(x))=\sig^{-1}(h(x))=\sig^{-1}(x)$.

Condition (2) implies that for every $h\in H$, $\phi(h)=0$, in other words
$f(h(x))=f(x)$ for all $x\in\Cal{X}$.
Hence $f$ factors through $\Cal{X}/H$, giving us a quotient curve $\Cal{X}/H$ over
$\bb{P}^1$ which maps to $\Cal{A}$, and so if we replace $(\Cal{X},G)$ by $(\Cal{X}/H,G/H)$
both conditions (1) and (2) are still satisfied. By the minimality of $|G|$, $H$ must be trivial.
So each $\sig\in G$ has no fixed point on $\Cal{X}$.
Hence we may assume that $G$ acts freely on $\Cal{X}$.


We define a  map $F:J(\Cal{X})\to \Cal{A}$ as follows; for $x_i, y_j \in \Cal{X}$,
\[
F\paren{\sum_{i=1}^n[x_i]-\sum_{i=1}^n[y_i]}=\sum_{i=1}^n f(x_i)-\sum_{i=1}^n f(y_i).
\]
The map $F$ is well defined because every morphism from a rational curve to an abelian variety is constant,
hence linearly equivalent divisors map to the same element of $\Cal{A}$, and it is
surjective because of the assumption that $f(\Cal{X})$ generates $\Cal{A}$.
Moreover, condition $(2)$ implies that for $\sig\in G$,
\[
 F\paren{\sum\limits_{i=1}^n[\sig(x_i)]-\sum\limits_{i=1}^n[\sig(y_i)]}=F\paren{\sum\limits_{i=1}^n[x_i]-\sum\limits_{i=1}^n[y_i]},
\]
so $F$ is $G$-equivariant. Therefore,  by identifying $J(\Cal{X})(\bb{C})$
with the quotient $\h_1(\Cal{X}(\bb{C}),\bb{R})/\h_1(\Cal{X}(\bb{C}),\bb{Z})$,
there exists a surjective morphism from $\h_1(\Cal{X}(\bb{C}),\bb{R})$ onto the Lie algebra $\Lie(\Cal{A})$ of $\Cal{A}$,
and this is $G$-equivariant by \cite[Ch.~6]{S}. Hence we have
$\dim(\Cal{A})\leq \dim (\h_1(\Cal{X}(\bb{C}),\bb{R}))=  \dim (\h_1(\Cal{X/G}(\bb{C}),\bb{R}))=g(\Cal{X}/G).$
\end{proof}

\

Now we are ready to prove Proposition~\ref{thm3}.

\begin{proof}[Proof of Proposition~\ref{thm3}]
Since $\Cal{X}$ is a cover of $\bbp^1_K$ over $K$, there exist infinitely many
degree $d$ divisors $D_Q$ on $\Cal{X}$ that are the preimages of $K$-rational
points $Q$ of $\bbp^1$ and are defined over $K$.

Fix one of $D_Q$ say, $D$. Then pick any point $P_Q$ on $\Cal{X}$ that lies
in the preimage of $Q$ (i.e., belongs to the support of $D_Q$) and consider
the degree zero divisor $d [P_Q]-D$ and (its linear equivalence class)
the corresponding point $\widetilde{P_Q}$ of the Jacobian variety $J(\Cal{X})$ of $\Cal{X}$,
which defines a $K$-rational non-constant morphism $f$ from $\Cal{X}$ to $J(\Cal{X})$,
i.e. $f(P):=d [P]-D.$

Now by varying $Q$ in $\bbp^1_K$ or by Lemma~\ref{lma1}
there are infinitely many points $\{P_i\}_i$ on $\Cal{X}$
such that $P_i\in \Cal{X}(L_i)\setminus \Cal{X}(K)$,
where $L_i$ are linearly disjoint degree $d$-extensions of $K$.  Let
\[
 D_i=d[P_i]-D
\]
denote the corresponding divisor in $J(\Cal{X})$.  Each divisor $D_i$  gives us a point on $J(\Cal{X})$
defined over each degree $d$-extension $L_i$ of $K$ via the $K$-rational map $f$.
By Lemma~\ref{lma3}, all but finitely many $D_i$ are non-torsion points of $J(\Cal{X})$.

Now we claim that for infinitely many $i$, $D_i$ are not defined over $K$.
For the Galois closure $\Cal{Y}$ of $\Cal{X}$ which is a Galois cover of $\bbp^1_K$,
let $G_0$ be a group isomorphic to each Galois closure of $L_i$ over $K$.
Let $\pi: \Cal{Y}\rightarrow \Cal{X}$ and $H\subseteq G_0$ such that $\Cal{Y}/H=\Cal{X}$. Let
$G:=\{\sig\in G_0 : f(\pi(\sig(y)))-f(\pi(y)) \text{ is a constant for all } y\in \Cal{Y}\}.$
Then $G$ is normal in $G_0$ (similarly as we have shown in the proof of Proposition~\ref{disjJac})
and $H\subseteq G$. Let $\Cal{X}':=\Cal{Y}/G$. Then we have a non-constant map $g$ from $\Cal{X}$ to $\Cal{X}'$.
Since the genus $g(\Cal{X})$ of $\Cal{X}$ is greater than $1$, we have that $g(\Cal{X}) > g(\Cal{X}')$
unless $G= H$, and $g(\Cal{X}) =g(\Cal{X}')$ if and only if $G=H$ (i.e. $g$ is an isomorphism).
On the other hand, since $f(\Cal{Y})$ generates $J(\Cal{X})$ via $f\circ \pi$, Proposition~\ref{disjJac}
implies that $g(\Cal{X})\leq g(\Cal{Y}/G)=g(\Cal{X}')$. Therefore, we conclude that $G=H$, i.e. $\Cal{X}=\Cal{Y}/G$.
This implies that for any element $\tau$ in $G_0 \setminus G$, $f(\tau(x))-f(x)$ is not constant for all $x\in \Cal{X}$.
Hence there are infinitely many $i$ such that $f(x_i)=D_i$ are not defined over $K$.

Hence there are infinitely many non-torsion points $D_i\in J(\Cal{X})(L_i) \setminus J(\Cal{X})(K)$.
Moreover, by Lemma~\ref{lma3}, there are infinitely many $i$ such that $\sig(D_i)-D_i$ are not torsion for all $\sig\in G_0$.  In order to show that they are linearly independent, if
\begin{equation}\label{dep}
n_1D_1+n_2D_2+\cdots n_kD_k=O, \text{for some integers } n_i,
\end{equation}
for each $i$, there exists $\sig_i\in \Gal(L_1L_2\cdots L_k/K)$ such that $\sig_i(D_i)\neq D_i$ but $\sig_i(D_j)=D_j$ for all $j\neq i$ by the linear disjointness of $L_i$. Then by applying $\sig_i$ to Eq. (\ref{dep}) and subtracting one from another, we get
$n_i(\sig(D_i)-D_i)=O,$
which leads a contradiction. Hence the rank of $J(\Cal{X})$ gains over each $L_i$.
\end{proof}

%
%
%

\begin{proof}[Proof of Proposition~\ref{thm4}]
 Every curve $\Cal{X}$ is bi-rationally equivalent to a plane curve, allowing
 for singularities. Irreducibility, however, is preserved.  Let $\Cal{C}$
 be such a plane curve, and suppose it is given by a homogeneous equation
 \[
  F(X_0:X_1:X_2)=0
 \]
where we define $n=\deg F$.  We now make a change of variables as follows.
Let $L/K$ be a finite extension such that $\Cal{C}$ has a non-singular
point $P_1$ defined over $L$.  Then choose a second point have $P_2$,
not lying on $\Cal{C}$ such that the directional derivative
\[
 \nabla_{\overrightarrow{P_1P_2}}F(P_1)
\]
does not vanish (i.e. $P_2$  does not lie on the tangent of $\Cal{C}$
at $P_1$). The points $P_1$ and $P_2$ together define a line in $\bb{P}^2$;
let $P_0$ be a point not on this line. Finally if $P_i=(a_{i0}:a_{i1}:a_{i2})$
are $X_0,X_1,X_2$-coordinates, then
\[
 X_i=a_{i0}U_0+a_{i1}U_1+a_{i2}U_2
\]
defines a change of variables, such that $F(U_0:U_1:U_2)$ has the following
properties:
\begin{enumerate}
 \item The coefficient of $U_1^n$-term vanishes,
 \item The coefficient of $U_1^{n-1}U_2$-term does not vanish.
\end{enumerate}
In the new coordinates, $F$ is defined over $L$, however there is enough
freedom to the choice of $P_2$ and $P_0$ so that if $a_{1j}\in L\rcomp K$,
then $a_{2j}$ and $a_{0j}$ can be chosen such that
\[
 \frac{a_{2j}}{a_{1j}},\frac{a_{0j}}{a_{1j}}\in K.
\]
We can then re-scale by the substitution $V_j=a_{1j}U_j$. On the
other hand if $a_{1j}\in K$ then we simply take $a_{2j},a_{0j}\in K$
and $V_j=U_j$. The combined effect is that we obtain a projective
linear transformation of $\bb{P}^2$ defined over $K$, such that
$F(V_0:V_1:V_2)$ has the same two properties as $F(U_0:U_1:U_2)$ above.
Dividing $F(V_0:V_1:V_2)$ by $V_0^n$ and making the substitution
\[
 v_1=\frac{V_1}{V_0}\quad\text{and}\quad v_2=\frac{V_2}{V_0}
\]
gives us an equation $f(v_1,v_2)=0$ defining $\Cal{C}$.  By construction
the polynomial $f$ is defined over $K$ and irreducible. We now construct
an auxiliary curve $\Cal{C}_1$ from $f$ by the substitution.
\[
 v_1=s+t^{k_1}\quad \text{and}\quad v_2=b+t^{k_2}
\]
for suitable $k_1,k_2\in\bb{Z}^+$ and $b\in K$.
The polynomial $g(s,t)$ defining $\Cal{C}_1$ must be designed to
be irreducible and have degree $d=k_1(n-1)+k_2$.  Since $f$ is irreducible,
by Hilbert's irreducibility theorem there exists $b\in K$ such that $f(v_1,b)$
is irreducible.  Under the substitution we have
\[
 f(s,b)=g(s,0),
\]
and so $g(s,t)$ must be irreducible with this choice of $b$.

To obtain the desired degree, we plot the exponents $(n_i,n_j)$
of all non-vanishing terms $c_{ij}v_i^{n_i}v_j^{n_j}$ of $f$ in a plane.
Now consider the region $R$ formed by the convex hull of these points.
The two properties of mentioned above imply that $(n-1,1)$
is a corner point of $R$, hence by the graphical method of solving
a linear programming problem
\[
k_1n_1+k_2n_2
\]
will attain its maximum value for the region $R$ uniquely
at the point $(n-1,1)$ if  the slope of the line
$k_1n_1+k_2n_2=0$  is less than $-1$.  Since the slope
is $-k_1/k_2$, this means we need $k_1> k_2$.

Now we need to determine what values of $d=k_1(n-1)+k_2$
are possible, with positive integers $k_1,k_2$ satisfying $k_1> k_2$.
If we increase $k_1$ by 1, the value of $d$ goes up $n-1$, and
this gap can be filled in by $k_2$ if and only if $k_2$ is allowed
to range over $n-1$ consecutive integers. This makes $n$ the minimum
value for $k_1$ and hence we obtain:
\[
 N=n(n-1)+1.
\]
Applying Proposition~\ref{thm3} to the curve $\Cal{C}_1$ where the covering
of $\bb{P}^1$ is obtained by taking the $s$-coordinate, completes the
proof.
\end{proof}
\begin{rmk}
 In some cases it is possible to get a better value for $N$.  For example,
 for certain curves we may be allowed to take $k_1=k_2$.  The maximum
 value may no longer be uniquely obtained at $(n-1,1)$, but if the cancellation
 of the highest degree terms of $f$ does not occur, this maximum may
 still be the degree. In that case we are allowed to take $N=(n-1)^2+1$.

 Also, if we are only interested in prime degree, and pay attention
 to the prime gaps, we can do better. For example, in the case of
 the corollary for elliptic curves, we have $n=3$ giving us $N=7$
 at least. But $5$ is also obtained by $k_1=2$ and $k_2=1$, and
 we can get $2$ and $3$ by applying Proposition~\ref{thm3} directly
 to the elliptic curve.
\end{rmk}
Next, we consider the special case of a degree $p$-Galois covering
that is ramified at 0, 1, and $\infty$.  By taking $p=7$, this includes
the Klein quartic as a special case.  There are a couple of interesting
things to note about this case.  The first is that when the construction
is done over $K/\bb{Q}(\zeta)$, where $\zeta$ is a primitive 7-th root of unity,
the extensions $L/K$ are cyclic Galois extensions.  The second is that by
Remark~\ref{Special}, for $K=\bb{Q}$
or $\bb{Q}(\zeta)$ the only $K$-rational points which don't lift to a
point lying in a degree 7 extension are 0, 1, and $\infty$.  For $K=\bb{Q}(\zeta)$
we use Hilbert's correction of Kummer's proof in the case of regular primes
(see \cite{Grosswald}).

\begin{proof}[Proof of Proposition \ref{thm1}]
First we investigate the situation with an automorphism of $\Cal{X}$
that cyclically permutes the points 0, 1, and $\infty$.
Given loops $\gamma_0,\gamma_1,\gamma_\infty$, about the points
0, 1, $\infty$ respectively each with the same winding number $Np$
for some positive integer $N$, these loops will lift to complete loops in $\Cal{X}$
with winding numbers $w_0$, $w_1$, $w_\infty$, which must satisfy
\begin{equation}\label{eq:1}
 w_0+w_1+w_\infty\equiv 0\mod p.
\end{equation}
Now suppose that $\Cal{X}$ has an automorphism that cyclically permutes
the points 0, 1, and infinity.  If we take $N$ to be a common multiple
of the minimal values for $w_0$, $w_1$, $w_\infty$, then the lifts
of $\gamma_0,\gamma_1,\gamma_\infty$ will remain complete
loops under the automorphism.  Furthermore, the automorphism
causes $w_0$, $w_1$, and $w_\infty$ to be multiples of each other
mod $p$, say
\[
w_1\equiv c_0 w_0 \bmod p,\quad
w_\infty \equiv c_1 w_1 \bmod p,\quad
w_0\equiv c_2w_\infty \bmod p
\]
for some integers $c_0,c_1,c_2$, and in fact it can be shown that $c_0$, $c_1$, and $c_2$ must be the same $\bmod\, p$,
hence \prettyref{eq:1} reduces to
\begin{equation}\label{eq:5}
  1+n+n^2\equiv 0\mod p
\end{equation}
for some positive integer $n$ which can be taken less than $p$. In terms of equations, this means
we have the affine model defined by
\begin{equation}\label{eq:2}
  y^p=x^n(x-1).
\end{equation}
For $p=7$ this agrees with \cite[equation 2.2]{Elkies} up to a change of variables.
From the quadratic reciprocity we see that \prettyref{eq:5} has a solution when $p=3$
or $p\equiv 1 \mod 6$. However, in the case $p=3$, equation \pref{eq:2} defines an elliptic
curve, and the quotient by the cyclic permutation gives us an isogeny, so in
this case we do not really get anything beyond what Proposition~\ref{thm3} already states.

Whenever a solution to \prettyref{eq:5} exists, we can show that $\Cal{X}$ has an
automorphism that cyclically permutes 0, 1, and $\infty$, in the following way.
We do this by considering an auxiliary curve $\Cal{Y}$, having a map to $\Cal{X}$,
and actually, for certain values of $p$, the curve $\Cal{Y}$ is actually a
non-singular model for $\Cal{X}$. In particular we define $\Cal{Y}$ as follows: for a given integer $m\geq 0$,
\begin{equation}\label{eq:6}
  X^mY+Y^mZ+Z^mX=0.
\end{equation}
It is non-singular and has an automorphism that cyclically permutes the points
\[
 (1:0:0),\, (0:1:0),\, (0:0:1).
\]
For convenience, let
$S:=\{(1:0:0),\, (0:1:0),\, (0:0:1)\}.$
The fixed points of this automorphism are
\begin{equation}\label{eq:7}
 (\rho:\rho^2:1)\quad\text{and}\quad(\rho^2:\rho:1)
\end{equation}
where $\rho$ is a primitive cube root of unity. These do not always lie
on the curve, and as we shall see this is the difference between the
cases $p=3$ and $p\equiv 1\mod 6$. Following Elkies \cite{Elkies},
we set up a map from $\Cal{Y}\rcomp S$ to the line $a+b+c=0$ in $\bb{P}^2$
in the obvious way:
\[
 \psi:(X:Y:Z)\mapsto (X^mY:Y^mZ:Z^mX).
\]
This map can be extended to all of $\Cal{Y}$ by defining
\begin{align*}
 \vphi:(1:0:0)\mapsto (1:0:-1),\\
 \vphi:(0:1:0)\mapsto (-1:1:0),\\
 \vphi:(0:0:1)\mapsto (0:-1:1),
\end{align*}
and it is easy to check that when $\psi$ is extended in this way, it remains continuous.
If $(a:b:c)$ is a point in the image of this map $\vphi$, it is easy to verify that
\[
  \paren{\frac{Y}{Z}}^{m^2-m+1}=\frac{ab^{m-1}}{c^m}.
\]
By making the substitution $u=-\frac{b}{c}$ and $v=(-1)^{m-1}Y/Z$, this give us
\begin{equation}
v^{m^2-m+1}=u^{m-1}(u-1).
\end{equation}
Equation~\prettyref{eq:2} can be obtained from this by the additional substitutions
$x=u$, $y=v^{(m^2-m+1)/p}$ and $n=m-1$, which essentially amounts to taking
a quotient.

Now it be calculated that the automorphism of $\Cal{Y}$
that cyclically permutes the variables $X,Y,Z$ passes to an automorphism of $\Cal{X}$
permuting $0,1,\infty$, and when $p\equiv 1\mod 6$ the two fixed points in~\prettyref{eq:7}
give us two fixed points in the model \prettyref{eq:2} with $x$ and $y$ coordinates
in the form $\pm \rho^i$ for $i=1,2$. For $p=3$, they do not lie on the curve. Hence by applying the
Riemann-Hurwitz formula with this ramification data, we find that the genus of the quotient is
\[
 g=\begin{cases}
    \frac{p-1}{6}&\text{if }p\equiv1\mod 6\\
    1&\text{if }p=3.
   \end{cases}
\]
Any $K$-rational point on $\Cal{X}$ maps to a $K$-rational point on the quotient,
so the rest of the argument is the same as the proof of Proposition~\ref{thm3}.
\end{proof}

To prove Theorem \ref{thm2} we need the concept of a generic polynomial.
The book \cite{Ledet} by Jensen, Ledet, and Yui  gives a good introduction to the
subject. In particular, we will use following lemma to construct
a generic polynomial for a cyclic  degree 3-extension of $\bb{Q}$.

\begin{lma}\label{lma2}
 Let $f$ be a cubic polynomial that is irreducible over a number field $K$, and let $d(f)$ be
its discriminant. Then
 \[
  \Gal(f/K)\simeq\begin{cases}
                  S_3&\text{if }d(f)\notin (K^\times)^2,\\
                  C_3&\text{if }d(f)\in (K^\times)^2,
                 \end{cases}
 \]
 where $S_3$ is the symmetric group on $3$ letters and $C_3$ is the cyclic group of order $3$.
\end{lma}
The point is that if we allow the coefficients of $f$ to have parameters,
and we use Hilbert's irreducibility theorem to obtain values of those parameters
such that $f$ remains irreducible over the ground field, then this theorem allows the
Galois group to be determined. The conditions for this turn out to be
much weaker than those for a covering of $\bb{P}^1$ to be Galois,
and this is what enables us to prove Theorem~\ref{thm2}.

\begin{proof}[Proof of Theorem \ref{thm2}]
Consider the polynomial
\begin{equation}\label{eq:4}
 f(x,t)=x^3+(a_4-a_1t)x+(a_6-a_3t-t^2)
\end{equation}
where $a_4,a_1,a_6,a_3\in\bb{Q}$ are to be determined.
For convenience, we assume that
\[
a_3=a_1\frac{a_6}{a_4}-\frac{a_4}{a_1},
\]
so that $a_6-a_3t-t^2$ contains $a_4-a_1t$ as a factor.
Then the discriminant is
\begin{equation}\label{eq:3}
  d(f)=\paren{-4(a_4-a_1 t)-27\paren{\frac{a_6}{a_4}+\frac{t}{a_1}}^2}(a_4-a_1t)^2
\end{equation}
Our strategy is to use the Diophantine methods to find the values of $t$ making
$d(f)$ a square in $\bb{Q}$.  While $f(x,t)$ gives us a covering, it does not give us a generic
polynomial with $C_3$ as the Galois group.  Effectively the Diophantine approach
replaces $t$ with an auxiliary variable $s$ such that $f(x,s)$ becomes
a generic polynomial with $C_3$ as the Galois group.

By completing the square, the first factor in equation \pref{eq:3} satisfies that
\begin{equation}
\begin{split}
 \notag -4(a_4-a_1 t)&-27\Big(\frac{a_6}{a_4}+\frac{t}{a_1}\Big)^2 \\
&=4\paren{\frac{a_1^4}{27}-a_1^2\frac{a_6}{a_4}-a_4}-27\paren{\frac{t}{a_1}+\frac{a_6}{a_4}-\frac{2 a_1^2}{27}}^2.
\end{split}
\end{equation}
If we set the first term $4\paren{\frac{a_1^4}{27}-a_1^2\frac{a_6}{a_4}-a_4}$ equal to 1, this reduces the number of free variables by
1 giving us
\[
 a_6=\frac{a_4 a_1^2}{27}-\frac{a_4}{4a_1^2}-\frac{a_4^2}{a_1^2}.
\]
The Diophantine equation
\[
 1=u^2+3v^2
\]
can be solved in the same way as the Pythagorean case.  In
particular we have
\[
 u=\frac{1-3s^2}{1+3s^2}\quad\text{and}\quad
 v=3\paren{\frac{t}{a_1}+\frac{a_6}{a_4}-\frac{2 a_1^2}{27}}=\frac{2s}{1+3s^2}.
\]
This parameterization allows $t$ to be replaced by a rational expression in
$s$ such that the discriminant of $f(x,s)$ is
\[
 d(f)=u^2(a_4-a_1t)^2.
\]
Applying Hilbert's irreducibility theorem gives us infinitely many values $s\in\bb{Q}$ such
that $f(x,s)$ is irreducible over $\bb{Q}$, and by lemma \ref{lma2}, the roots of any such polynomial
will generate a cyclic degree 3 extension of $\bb{Q}$.  In terms of the
elliptic curve $E$ defined by $f(x,t)=0$, where $f(x,t)$ is given in~\pref{eq:4}, this means that
the specific $t$ values we get from rational $s$ values give us points $(x,t)$ on $E$ where the $x$-coordinate
lies in a cyclic degree 3-extension $L/\bb{Q}$. Since changing the value of $s$ gives us infinitely
many different values of $t$, we obtain infinitely
many cyclic degree 3-extensions $L/\bb{Q}$ over which $E$ gains rank. The linear
disjointness also follows in the same way as in the proof of Lemma~\ref{lma1}.

As for which elliptic curves this occurs for, it would seem that the variables $a_1$
and $a_4$ remain undetermined.  However, when computing the $j$-invariant, the
variable $a_4$ cancels out and we are left with an expression depending only on $a_1$,
namely
\begin{equation}\label{j}
 j=256\frac{(a_1^4 + 54)^3a_1^4}{(4a_1^4 - 27)^3}.
 \end{equation}
\end{proof}

\begin{rmk}\label{eta} Recall that the function $f$ defined by 
$$f(z)=\Big(\frac{\eta(z)}{\eta(3z)}\Big)^2=q^{-1}+15+54q-76q^2-243 q^3+1188 q^4-\cdots$$
is a modular function on $\Gamma_0(3)$. And the $j$-invariant given in $(\ref{j})$ can be paramatrized as follows:
Letting $a_1^4=\frac{f}{4}$, $$j=\frac{f(f+216)^3}{(f-27)^3}.$$
We deeply thank Professor Masanobu Kaneko for letting us this observation.
\end{rmk}


\begin{thebibliography}{99}

\bibitem{Bosch} Bosch, Siegfried and L{\"u}tkebohmert, Werner and Raynaud,
              Michel, \emph{N\'eron models}, {Ergebnisse der Mathematik und ihrer Grenzgebiete (3) [Results
              in Mathematics and Related Areas (3)]}, Vol. 21, Springer-Verlag, Berlin, 1990.
\bibitem{Elkies} Elkies, Noam D., {\emph The {K}lein quartic in number theory}, {The eightfold way}, {Math. Sci. Res. Inst. Publ.}, Vol. 35 (1999), Cambridge Univ. Press, Cambridge, 51--101.

\bibitem{Kuwata1} Fearnley, Jack and Kisilevsky, Hershy and Kuwata, Masato, \emph{Vanishing and non-vanishing {D}irichlet twists of
              {$L$}-functions of elliptic curves}, {J. Lond. Math. Soc. (2)}, $86 (2)$ (2012), 539--557.

\bibitem{Grosswald} Grosswald, Emil, \emph{Topics from the theory of numbers}, {Modern Birkh\"auser Classics}, {Reprint of the 1984 second edition [MR0807527]}, Birkh\"auser Boston, Inc., Boston, MA, 2009.

\bibitem{Ledet}{Jensen, Christian U. and Ledet, Arne and Yui, Noriko}, \emph{Generic polynomials}, {Mathematical Sciences Research Institute Publications}, {Constructive aspects of the inverse Galois problem,} Vol. 45,
 {Cambridge University Press, Cambridge}, {2002}.

\bibitem{Costa}{Mendes da Costa, Dave},
\emph{On ranks of Jacobian varieties in prime degree extensions}, {Acta Arith.}, Vol. 161, no.3 (2013), 241--248.

\bibitem{S}{Serre, Jean-Pierre}, \emph{Local fields},

\bibitem{Serre}{Serre, Jean-Pierre}, \emph{Topics in {G}alois theory}, {Research Notes in Mathematics}, Vol. 1, {Lecture notes prepared by Henri Damon With a foreword by Darmon and the author}, 1992.

\bibitem{Silverman}{Silverman, Joseph H.}, \emph{Integer points on curves of genus {$1$}}, {J. London Math. Soc. (2)}, {\bf 28 (1)} (1983), 1--7.
\end{thebibliography}

\end{document}